\documentclass{amsart}

\usepackage{amsmath,amssymb,amsfonts,enumerate,amsthm,graphicx,color,hyperref}

\renewcommand{\dim}{\mbox{dim}\,}
\renewcommand{\dim}{\mbox{dim}\,}

\newcommand{\set}{\mbox{set}\,}

\newcommand{\T}{\mathrm}
\newcommand{\N}{\mathbb{N}}

\newtheorem{thm}{Theorem}[section]
\newtheorem{cor}[thm]{Corollary}
\newtheorem{lem}[thm]{Lemma}

\newtheorem{defn}[thm]{Definition}

\newtheorem{rem}[thm]{Remark}

\numberwithin{equation}{section}

\begin{document}
\bibliographystyle{amsplain}

\title[On vertex decomposable simplicial complexes]{On vertex decomposable simplicial complexes
and their Alexander duals}
\author[S. Moradi and F. Khosh-Ahang]{Somayeh Moradi and Fahimeh Khosh-Ahang}
\address{Somayeh Moradi, Department of Mathematics,
 Ilam University, P.O.Box 69315-516, Ilam, Iran and School of Mathematics, Institute
 for Research in Fundamental Sciences (IPM), P.O.Box: 19395-5746, Tehran, Iran.} \email{somayeh.moradi1@gmail.com}
\address{Fahimeh Khosh-Ahang, Department of Mathematics,
 Ilam University, P.O.Box 69315-516, Ilam, Iran.}
\email{fahime$_{-}$khosh@yahoo.com}


\begin{abstract}

\noindent In this paper we study the Alexander dual of a vertex decomposable simplicial complex. We
define the concept of a vertex splittable ideal
and show that a simplicial complex $\Delta$ is  vertex decomposable if and only if $I_{\Delta^{\vee}}$ is
a vertex splittable ideal. Moreover, the properties of vertex splittable ideals are studied. As the main result, it is proved that any vertex splittable ideal has a Betti splitting
and the graded Betti numbers of such ideals are explained with a recursive formula.
As a corollary, recursive formulas for the regularity and projective dimension
of $R/I_{\Delta}$, when $\Delta$ is a vertex decomposable simplicial complex, are given. Moreover, for a vertex decomposable graph $G$,
a recursive formula for the graded Betti numbers of its vertex cover ideal is presented. In special cases, this formula is explained, when $G$
is chordal or a sequentially Cohen-Macaulay bipartite graph.  Finally, among the other things, it is shown that an edge ideal of a graph is vertex splittable if and only if it has linear resolution.
\end{abstract}

\maketitle
\section*{Introduction}

There is a natural correspondence between squarefree monomial ideals in the  polynomial ring $R=k[x_1,\ldots,x_n]$ (over the field $k$)
and simplicial complexes with vertex set $\{x_1,\ldots,x_n\}$ via the Stanley-Reisner correspondence. In this regard Alexander duality plays an important role in the study of Stanley-Reisner rings. Finding Alexander dual concepts and translating a property of a simplicial complex $\Delta$ in the Alexander dual ideal $I_{\Delta^{\vee}}$, are important topics in combinatorial commutative algebra.
Eagon and Reiner in \cite{Eagon} introduced Alexander dual complexes and proved the following interesting result.

\begin{thm}\cite[Theorem 3]{Eagon}\label{EG}
A simplicial complex $\Delta$ is Cohen-Macaulay if and only if $I_{\Delta^{\vee}}$ has a linear resolution.
\end{thm}

Later in \cite{HD} and \cite{HH} shellable and sequentially Cohen-Macaulay complexes were characterized in terms of Alexander duals.
\begin{thm}\cite[Theorem 1.4]{HD}
$\Delta$ is a shellable simplicial complex if and only if $I_{\Delta^{\vee}}$ has linear quotients.
\end{thm}

\begin{thm}\cite[Theorem 2.1]{HH}
$\Delta$ is sequentially Cohen-Macaulay if and only if $I_{\Delta^{\vee}}$ is componentwise linear.
\end{thm}

Vertex decomposability is another topological combinatorial notion such as shella-blity, which is related to the
algebraic properties of the Stanley-Reisner ring of a simplicial complex. Provan and Billera  \cite{Provan+Billera} first introduced the notion of $k$-decomposable for a pure simplicial complex. For $k=0$, this notion is known as a vertex decomposable simplicial complex. This definition was extended to non-pure complexes by Bj\"{o}rner and Wachs in \cite{BW}.
Defined in a recursive manner, vertex decomposable simplicial complexes form a well-behaved class of simplicial complexes and have been studied in many papers.
In \cite{DE}, vertex decomposablity was used in an interesting way to study the algebraic properties of edge ideals and some nice results on edge ideals were obtained by combinatorial topological techniques. It was proved
that the independence complex of
a whiskered graph is a pure vertex decomposable simplicial complex, and hence,
Cohen-Macaulay.
See also \cite{W}, in this regard.
More nice results on vertex decomposable simplicial complexes may be found in \cite{BFHV, BV, CN, KM, Moradi, VT, W} and \cite{W1}. The following implications for a simplicial complex are known.
$$\textrm{Vertex decomposable} \Rightarrow \textrm{Shellable} \Rightarrow \textrm{Sequentially Cohen-Macaulay}.$$

So, inspired by the above results
in conjunction with the characterizations of shellable, sequentially Cohen-Macaulay and Cohen-Macaulay simplicial complexes by means of Alexander dual ideals, it is natural to ask if something similar can be said about vertex decomposable
simplicial complexes. In this paper, we seek a dual concept for this class of simplicial complexes. To this end, we introduce the notion of a vertex splittable ideal, which is shown to be an appropriate dual concept for vertex decomposability and then some nice properties of such ideals are  achieved. By exploiting this dual concept, we refine our results for flag complexes and
 study the minimal free resolution of vertex cover ideal of vertex decomposable graphs.

The paper is organized as follows.
In the first section,
we recall some definitions and theorems that we use
in the sequel. We begin Section 2 with the definition of vertex splittable ideals.  Then it is shown that
$\Delta$ is a vertex decomposable simplicial complex if and only if $I_{\Delta^{\vee}}$ is a vertex splittable ideal (see Theorem \ref{vI}).
Theorem \ref{q} illustrates that vertex splittable ideals have linear quotients. So, one may deduce
that a vertex splittable ideal generated by monomials in the same degrees has a linear resolution.  As another main result, in Theorem \ref{ss}, it is proved that a vertex splittable ideal has a Betti splitting
and so a recursive formula for its graded Betti numbers is presented (see Remark \ref{bettis}). Hence, it is shown that
for a vertex decomposable simplicial complex $\Delta$,
the regularity and projective dimension of $I_{\Delta}$ can be computed by inductive formulas.
 In Section 3, applications of Theorem \ref{ss} to vertex cover ideal of a vertex decomposable graph are given. In Theorem \ref{graphvd} the graded Betti numbers of vertex cover ideal of such graphs are explained by a recursive formula. For two families of vertex decomposable graphs, sequentially Cohen-Macaulay bipartite graphs and chordal graphs, this formula has been stated more precisely. Another application is Theorem
\ref{propchor} which states that if $G$ is a chordal graph, then $I(G^c)$ is a vertex splittable ideal. Using this fact and the characterization of Fr\"{o}berg for edge ideals of graphs with linear resolution, it is shown that an edge ideal $I(G)$ is vertex splittable if and only if
$I(G)$ has a linear resolution. In Corollary \ref{corchor1}, a simple argument is given to show that for a flag complex $\Delta_G$,  the Alexander dual simplicial complex  $\Delta_G^{\vee}$ is
vertex decomposable if and only if it is Cohen-Macaulay and these conditions hold if and only if $I(G)$ is a vertex splittable ideal.

\section{Preliminaries}

Throughout this paper, we assume that $X=\{x_1, \dots, x_n\}$, $\Delta$ is a simplicial complex on the vertex set $X$, $k$ is a field, $R=k[X]$ is the ring of polynomials in the variables $x_1, \dots, x_n$ and $I$ is a monomial ideal of $R$. For a monomial ideal $I$,
the unique set of minimal generators of $I$ is denoted by $\mathcal{G}(I)$.

In this section, we recall some preliminaries which are needed in the sequel. We begin with definition of a vertex decomposable simplicial complex. To this aim, we need to recall the definition of the \textbf{link} and the \textbf{deletion} of a face in $\Delta$.
For a simplicial complex $\Delta$ and $F\in \Delta$, the link of $F$ in
$\Delta$ is defined as $$\T{lk}_{\Delta}(F)=\{G\in \Delta: G\cap
F=\emptyset, G\cup F\in \Delta\},$$ and the deletion of $F$ is the
simplicial complex $$\T{del}_{\Delta}(F)=\{G\in \Delta: G\cap
F=\emptyset\}.$$

\begin{defn}\label{1.1}
{\rm A simplicial complex $\Delta$ is  called \textbf{vertex decomposable} if
$\Delta$ is a simplex, or $\Delta$ contains a vertex $x$ such that
\begin{itemize}
\item[(i)] both $\T{del}_{\Delta}(x)$ and $\T{lk}_{\Delta}(x)$ are vertex decomposable, and
\item[(ii)] any facet of $\T{del}_{\Delta}(x)$ is a facet of $\Delta$.
\end{itemize}
A vertex $x$ which satisfies condition (ii) is called a
\textbf{shedding vertex} of $\Delta$.}
\end{defn}

\begin{rem}
{\rm It is easily seen that $x$ is a shedding vertex of $\Delta$ if and only if no facet of $\T{lk}_{\Delta}(x)$ is a facet of $\T{del}_{\Delta}(x)$.}
\end{rem}

For a $\mathbb{Z}$-graded $R$-module $M$, the \textbf{Castelnuovo-Mumford regularity} (or simply regularity)
of $M$ is defined as
$$\T{reg}(M) := \max\{j-i: \ \beta_{i,j}(M)\neq 0\},$$
and the \textbf{projective dimension} of $M$ is defined as
$$\T{pd}(M) := \max\{i:\ \beta_{i,j}(M)\neq 0 \ \text{for some}\ j\},$$
where $\beta_{i,j}(M)$ is the $(i,j)$-th graded Betti number of $M$.

The notion of Betti splitting for monomial ideals was introduced in \cite{FHV} as follows.
\begin{defn} \cite[Definition 1.1]{FHV}
{\rm
Let $I$, $J$ and $K$ be monomial ideals in $R$ such that $\mathcal{G}(I)$ is the disjoint union of $\mathcal{G}(J)$ and $\mathcal{G}(K)$.
Then $I=J+K$ is a \textbf{Betti splitting} if $$\beta_{i,j}(I)=\beta_{i,j}(J)+\beta_{i,j}(K)+\beta_{i-1,j}(J\cap K),$$
for all $i\in \N$ and degrees $j$.}
\end{defn}

When $I=J+K$ is a Betti splitting, important homological invariants of $I$ are related to those invariants of the smaller ideals.

\begin{thm}(See \cite[Corollary 2.2]{FHV}.)\label{s}
Let $I=J+K$ be a Betti splitting. Then
\begin{itemize}
\item[(i)] $\T{reg}(I)=\max\{\T{reg}(J),\T{reg}(K),\T{reg}(J\cap K)-1\}$, and
\item[(ii)] $\T{pd}(I)=\max\{\T{pd}(J),\T{pd}(K),\T{pd}(J\cap K)+1\}$.
\end{itemize}
\end{thm}

For a squarefree monomial ideal $I=( x_{11}\cdots
x_{1n_1},\ldots,x_{t1}\cdots x_{tn_t})$, the \textbf{Alexander dual ideal} of $I$, denoted by
$I^{\vee}$, is defined as
$$I^{\vee}:=(x_{11},\ldots, x_{1n_1})\cap \cdots \cap (x_{t1},\ldots, x_{tn_t}).$$
For a simplicial complex $\Delta$ with vertex set $X$, the \textbf{Alexander dual simplicial complex} associated to $\Delta$ is defined as
$$\Delta^{\vee}=\{X\setminus F:\ F\notin \Delta\}.$$
For a subset $C\subseteq X$, by $x^C$ we mean the monomial $\prod_{x\in C} x$. The set of all facets of a simplicial complex $\Delta$
is denoted by $\mathcal{F}(\Delta)$.
One can see that
$$(I_{\Delta})^{\vee}=(x^{F^c} \ : \ F\in \mathcal{F}(\Delta)), $$
where $I_{\Delta}$ is the Stanley-Reisner ideal associated to $\Delta$ and $F^c=X\setminus F$.
Moreover, one can see that  $(I_{\Delta})^{\vee}=I_{\Delta^{\vee}}$.

Let $G$ be a graph with vertex set $V(G)$ and edge set $E(G)$. The edge ideal of $G$ is defined as the ideal $I(G)=(x_ix_j:\ \{x_i,x_j\}\in E(G))$. It is easy to see that $I(G)$ can be viewed as the Stanley-Reisner ideal of the simplicial complex
$$\Delta_G=\{F\subseteq V(G):\ e\nsubseteq F, \text{ for each } e\in E(G)\},$$ i.e., $I(G)=I_{\Delta_G}$. The simplicial complex $\Delta_G$ is called the \textbf{independence complex} of $G$.
Moreover, the Alexander dual of  $I(G)$ is called the \textbf{vertex cover ideal} of $G$.

A graded $R$-module $M$ is called
\textbf{sequentially Cohen--Macaulay} (over $k$) if there exists a
finite filtration of graded $R$-modules $$0=M_0\subset M_1\subset
\cdots \subset M_r=M$$ such that each $M_i/M_{i-1}$ is
Cohen--Macaulay and
$$\dim(M_1/M_0)<\dim(M_2/M_1)<\cdots<\dim(M_r/M_{r-1}).$$

A simplicial complex $\Delta$ is called sequentially Cohen--Macaulay if the Stanley-Reisner ring $R/I_{\Delta}$ is
sequentially Cohen--Macaulay. Also, we call a graph $G$ sequentially Cohen--Macaulay (resp. vertex decomposable), if $\Delta_G$ is a sequentially Cohen--Macaulay (resp. vertex decomposable) simplicial complex.

The following theorem, which was proved in \cite{T}, is one of our
main tools in the study of projective dimension and regularity of the ring $R/I_{\Delta}$.

\begin{thm}\cite[Theorem 2.1]{T} \label{1.3}
Let $I$ be a squarefree monomial ideal. Then
$\T{pd}(I^{\vee})=\T{reg}(R/I)$.
\end{thm}

\begin{defn}\label{1.2}
{\rm
A monomial ideal $I$ in the ring $R=k[x_1,\ldots,x_n]$ has \textbf{linear quotients} if there exists an ordering $f_1, \dots, f_m$ on the minimal generators of $I$ such that the colon ideal $(f_1,\ldots,f_{i-1}):(f_i)$ is generated by a subset of $\{x_1,\ldots,x_n\}$ for all $2\leq i\leq m$. We show this ordering by $f_1<\dots <f_m$ and we call it an order of linear quotients on $\mathcal{G}(I)$.}
\end{defn}

A monomial ideal $I$ generated by monomials of degree $d$ has a \textbf{linear resolution} if $\beta _{i,j}(I)=0$ for all $j\neq i+d$. Linear quotients is a strong tool to determine some classes of ideals with linear resolution. The main tool in this way is the following lemma.

\begin{lem}(See \cite[Lemma 5.2]{F}.)\label{Faridi}
Let $I=(f_1, \dots, f_m)$ be a monomial ideal with linear quotients such that all $f_i$'s are of the same degree. Then $I$ has a linear resolution.
\end{lem}

\section{Vertex decomposability and vertex splittable ideals}
In this section we study the ideal $I_{\Delta^{\vee}}$ for a vertex decomposable simplicial complex $\Delta$. We introduce the concept of a vertex splittable ideal and show that a simplicial complex $\Delta$ is vertex decomposable if and only if $I_{\Delta^\vee}$ is a vertex splittable ideal. Also, we prove that vertex splittable ideals have a Betti splitting and have linear
quotients. This gives
us information about some homological invariants of $I_{\Delta^\vee}$ such as Betti numbers.

\begin{defn}\label{dvs}
{\rm A monomial ideal $I$ in $R=k[X]$ is called \textbf{vertex splittable} if it can be obtained by the following recursive procedure.
\begin{itemize}
\item[(i)] If $u$ is a monomial and $I=(u)$, $I=(0)$ or $I=R$, then $I$ is a vertex splittable ideal.
\item[(ii)] If there is a variable $x\in X$ and vertex splittable ideals $I_1$ and $I_2$ of $k[X\setminus \{x\}]$ so that $I=xI_1+I_2$, $I_2\subseteq I_1$ and $\mathcal{G}(I)$ is the disjoint union of $\mathcal{G}(xI_1)$ and $\mathcal{G}(I_2)$, then $I$ is a vertex splittable ideal.
\end{itemize}
With the above notations if $I=xI_1+I_2$ is a vertex splittable ideal, then $xI_1+I_2$ is called a \textbf{vertex splitting} for $I$ and $x$ is called a \textbf{splitting vertex} for $I$.}
\end{defn}

Recently, the notion of a $k$-decomposable ideal was introduced in \cite{RY} and it was proved that it is the dual concept for k-decomposable simplicial complexes. In the case
$k=0$, $0$-decomposable simplicial complexes are precisely vertex decomposable simplicial complexes. Considering $k=0$ in \cite[Definition 2.3]{RY}, one can see that a shedding monomial $u$ in a $0$-decomposable ideal $I$
necessarily should satisfy the property $I_u\neq (0)$, where $I_u=(v\in \mathcal{G}(I):\ [u,v]=1)$ and $[u,v]$ is the greatest common divisor of $u$ and $v$, while in Definition \ref{dvs}, it is not the case, i.e., the way that a monomial ideal
splits in Definition \ref{dvs} is different from one in \cite[Definition 2.3]{RY}. For example let $I=(xx_1,\ldots,xx_n)$. Then
$I=x(x_1,\ldots,x_n)$. Setting $I_1=(x_1,\ldots,x_n)$ and $I_2=(0)$  with the same notations of the above definition, it is easy to see that $x$ is a splitting vertex for $I$, while it is not a shedding monomial in the sense of \cite[Definition 2.1]{RY}, since $I_x=(0)$.

First we prove the following lemma.
\begin{lem}\label{v}
Let $\Delta$ be a simplicial complex on the set $X$, $x\in X$  be a shedding vertex of $\Delta$, $\Delta_1=\T{del}_{\Delta}(x)$
and $\Delta_2=\T{lk}_{\Delta}(x)$. Then $$I_{\Delta^{\vee}}=xI_{\Delta_1^{\vee}}+I_{\Delta_2^{\vee}} \ \T{and} \ I_{\Delta_2^{\vee}}\subseteq I_{\Delta_1^{\vee}}.$$
\end{lem}
\begin{proof}
Let $\mathcal{F}(\Delta)=\{F_1,\ldots,F_m\}$ and $X'=X\setminus\{x\}$. We have $I_{\Delta^{\vee}}=(x^{X\setminus F_1},\ldots,x^{X\setminus F_m})$. Without loss of generality assume that $F_1,\ldots,F_k$ $(k\leq m)$ are all the facets of $\Delta$ containing $x$. Since $x$ is a shedding vertex of $\Delta$, one can see that
$\Delta_1=\langle F_{k+1},\ldots,F_m\rangle$. Also $\Delta_2=\langle F_1\setminus \{x\},\ldots,F_k\setminus \{x\}\rangle$. Therefore  $I_{\Delta_1^{\vee}}=(x^{X'\setminus F_{k+1}},\ldots,x^{X'\setminus F_m})$ and any generator of $I_{\Delta_2^\vee}$ is of the form $x^{X'\setminus (F_i\setminus \{x\})}=x^{X\setminus F_i}$ for any $1\leq i\leq k$. So $I_{\Delta_2^{\vee}}=(x^{X\setminus F_1},\ldots,x^{X\setminus F_k})$. Thus $I_{\Delta^{\vee}}=xI_{\Delta_1^{\vee}}+I_{\Delta_2^{\vee}}$.

To prove the last assertion, let $x^{X'\setminus (F_i\setminus \{x\})}$ be a generator of $I_{\Delta_2^{\vee}}$ for some
$1\leq i\leq k$. Since $F_i\setminus \{x\}$ is not a facet of $\T{del}_{\Delta}(x)$, there is a facet $F_j$ of $\T{del}_{\Delta}(x)$ such that
$F_i\setminus \{x\}\subseteq F_j$. Hence $x^{X'\setminus (F_i\setminus \{x\})}\in (x^{X'\setminus F_j})\subseteq I_{\Delta_1^{\vee}}$ which ensures that $I_{\Delta_2^{\vee}}\subseteq I_{\Delta_1^{\vee}}$.
\end{proof}

The following theorem characterizes when $\Delta$ is vertex decomposable
in terms of the Alexander dual of $I_{\Delta}$.

\begin{thm}\label{vI}
A simplicial complex $\Delta$ is vertex decomposable if and only if $I_{\Delta^\vee}$ is a vertex splittable ideal.
\end{thm}
\begin{proof}
Assume that $\Delta$ is a vertex decomposable simplicial complex with vertex set $X$. We use induction on $n=|X|$. If $n=1$, then $I_{\Delta^\vee}$ is clearly vertex splittable. Suppose inductively that the result has been proved for smaller values of $n$. In view of Lemma \ref{v}, we have $I_{\Delta^{\vee}}=xI_{\Delta_1^{\vee}}+I_{\Delta_2^{\vee}}$ and $I_{\Delta_2^{\vee}}\subseteq I_{\Delta_1^{\vee}}$, where $x$ is a shedding vertex of $\Delta$, $\Delta_1=\T{del}_{\Delta}(x)$
and $\Delta_2=\T{lk}_{\Delta}(x)$. Since $\Delta_1$ and $\Delta_2$ are vertex decomposable simplicial complexes on $X\setminus \{x\}$, inductive hypothesis implies that $I_{\Delta_1^{\vee}}$ and $I_{\Delta_2^{\vee}}$ are vertex splittable ideals in $k[X\setminus \{x\}]$.
 This completes the \emph{only if} part.

To prove the \emph{if} part, let $\Delta=\langle F_1, \dots, F_m\rangle$. Then $I_{\Delta^{\vee}}=(x^{X\setminus F_1},\ldots,x^{X\setminus F_m})$. If $I_{\Delta^{\vee}}=(u)$ for some monomial $u$, $I_{\Delta^{\vee}}=(0)$ or $I_{\Delta^{\vee}}=R$, then $\Delta$ is either a simplex or empty simplicial complex which is vertex decomposable. Otherwise, there is a variable $x\in X$ and vertex splittable ideals $I_1$ and $I_2$ of $k[X\setminus \{x\}]$ so that $I_{\Delta^{\vee}}=xI_1+I_2$, $I_2\subseteq I_1$. Suppose inductively that the result is true for any vertex splittable ideal in $k[X']$ with $|X'|<|X|$. We show that $x$ is a shedding vertex of $\Delta$, $I_1=I_{{\T{del}_\Delta(x)}^\vee}$ and  $I_2=I_{{\T{lk}_\Delta(x)}^\vee}$. Without loss of generality assume that $F_1,\ldots,F_k$ $(k\leq m)$ are all the facets of $\Delta$ containing $x$. Set $X'=X\setminus \{x\}$. Then
\begin{align*}
I_{\Delta^{\vee}}&=(x^{X\setminus F_1},\dots,x^{X\setminus F_k})+(x^{X\setminus F_{k+1}},\dots,x^{X\setminus F_m})\\
 &=(x^{X'\setminus (F_1\setminus \{x\})},\dots,x^{X'\setminus (F_k\setminus \{x\})})+x(x^{X'\setminus F_{k+1}},\dots,x^{X'\setminus F_m}).
\end{align*}
Hence $I_1=(x^{X'\setminus F_{k+1}},\dots,x^{X'\setminus F_m})$ and $I_2=(x^{X'\setminus (F_1\setminus \{x\})},\dots,x^{X'\setminus (F_k\setminus \{x\})})$.
Since $I_2\subseteq I_1$, one can see that for any $1\leq i\leq k$, there is an integer $k+1\leq j\leq m$ such that $F_i\setminus \{x\}\subseteq F_j$.
This ensures that $\T{del}_\Delta(x)=\langle F_{k+1}, \dots, F_m\rangle$ and $I_1=I_{{\T{del}_\Delta(x)}^\vee}$.
Also, clearly $\T{lk}_\Delta(x)=\langle F_1\setminus \{x\}, \dots, F_k\setminus \{x\}\rangle$, $I_2=I_{{\T{lk}_\Delta(x)}^\vee}$ and every facet of $\T{del}_\Delta(x)$ is a facet of $\Delta$.
By induction hypothesis $\T{del}_\Delta(x)$ and $\T{lk}_\Delta(x)$ are vertex decomposable. Hence, $\Delta$ is vertex decomposable as desired.
\end{proof}
In the following, it is shown that vertex splittable ideals have linear quotients. Note that when $I$ is a squarefree vertex splittable ideal, the following theorem implies the known fact that any vertex decomposable simplicial complex is shellable.
\begin{thm}\label{q}
Any vertex splittable ideal has linear quotients.
\end{thm}
\begin{proof}
Let $I$ be a vertex splittable ideal. If $I=(u)$ for some monomial $u$, $I=(0)$ or $I=R$, then clearly I has linear quotients.
Otherwise, there is a variable $x\in X$ and there are vertex splittable ideals $I_1$ and $I_2$ of $k[X\setminus \{x\}]$ such that $I=xI_1+I_2$ and $I_2\subseteq I_1$. By induction on $n=|X|$ we prove the assertion. For $n=1$ the result is clear.
By induction assume that for any set $X'$ with $|X'|<n$, any splittable ideal of the ring  $k[X']$ has linear quotients.
Thus $I_1$ and $I_2$ have linear quotients.
Let $f_1<\cdots<f_r$ and $g_1<\cdots<g_s$ be the order of linear quotients on the minimal generators of $I_1$
and $I_2$, respectively. We claim that the ordering $xf_1<\cdots<xf_r<g_1<\cdots<g_s$ is an order of linear quotients
on the minimal generators of $I$. For any $1\leq i \leq r$, the ideal $(xf_1,\ldots,xf_{i-1}):(xf_i)=(f_1,\ldots,f_{i-1}):(f_i)$ is generated by variables by assumption. For an integer $1\leq i\leq s$, let $(g_1,\ldots,g_{i-1}):(g_i)=(x_{i_1},\ldots,x_{i_t})$ and consider the ideal $(xf_1,\ldots,xf_r,g_1,\ldots,g_{i-1}):(g_i)$. We know that $x$ divides any generator of the colon ideal $(xf_j):(g_i)$ for any
$1\leq j\leq r$. Since $I_2\subseteq I_1$, there is an integer $1\leq \ell\leq r$ such that $g_i\in (f_\ell)$. Therefore
$(xf_\ell):(g_i)=(x)$.
It means that $(xf_1,\ldots,xf_r,g_1,\ldots,g_{i-1}):(g_i)=(x,x_{i_1},\ldots,x_{i_t})$. Thus $I$ has linear quotients.
\end{proof}

The following corollary is an immediate consequence of Lemma \ref{Faridi} and Theorem \ref{q}.
\begin{cor}\label{corlinear}
Let $I$ be a vertex splittable ideal generated by monomials in the same degrees. Then $I$ has linear resolution.
\end{cor}

The next theorem, which is a special case of \cite[Corollary 2.7]{Leila}, is our main tool to prove Theorem \ref{ss}. First we need to recall the following definition.

\begin{defn}\label{1.20}
{\rm
Let $I$ be a monomial ideal with linear quotients and $f_1<\dots <f_m$ be an order of linear quotients on the minimal generators of $I$. For any $1\leq i\leq m$, $\set_I(f_i)$ is defined as
$$\set_I(f_i)=\{x_k:\ x_k\in (f_1,\ldots, f_{i-1}) : (f_i)\}.$$}
\end{defn}

\begin{thm}\cite[Corollary 2.7]{Leila}\label{Leila}
Let $I$ be a monomial ideal with linear quotients with the ordering $f_1<\cdots<f_m$ on the minimal generators of $I$.
Then $$\beta_{i,j}(I)=\sum_{\deg(f_t)=j-i} {|\set_I(f_t)|\choose i}.$$
\end{thm}

Now, we are ready to bring one of the main results of this paper.
\begin{thm}\label{ss}
Let $I=xI_1+I_2$ be a vertex splitting for the monomial ideal $I$. Then $I=xI_1+I_2$ is a Betti splitting.

\end{thm}
\begin{proof}
By Theorem \ref{q}, $I$, $I_1$ and $I_2$ have linear quotients. Let $f_1<\cdots<f_r$ and $g_1<\cdots<g_s$ be the order of linear quotients on the minimal generators of $I_1$
and $I_2$, respectively. As it was shown in the proof of Theorem \ref{q} the ordering $xf_1<\cdots<xf_r<g_1<\cdots<g_s$ is an order of linear quotients
on the minimal generators of $I$, $\set_I(xf_t)=\set_{I_1}(f_t)$ for any $1\leq t\leq r$ and $\set_I(g_k)=\{x\}\cup \set_{I_{2}}(g_k)$ for any $1\leq k\leq s$. By Theorem \ref{Leila},
$$\beta_{i,j}(I)=\sum_{\deg(f_t)=j-i-1}{|\set_I(xf_t)|\choose i}+
\sum_{\deg(g_k)=j-i}{|\set_I(g_k)|\choose i}.$$
Thus
$$\beta_{i,j}(I)=\sum_{\deg(f_t)=j-i-1}{|\set_{I_{1}}(f_t)|\choose i}+
\sum_{\deg(g_k)=j-i}{|\set_{I_{2}}(g_k)|+1\choose i}.$$ Applying the equality
$${|\set_{I_{2}}(g_k)|+1\choose i}={|\set_{I_{2}}(g_k)|\choose i}+{|\set_{I_{2}}(g_k)|\choose i-1},$$
we have
$$\sum_{\deg(g_k)=j-i}{|\set_{I_{2}}(g_k)|+1\choose i}=\sum_{\deg(g_k)=j-i}{|\set_{I_{2}}(g_k)|\choose i}+\sum_{\deg(g_k)=j-i}{|\set_{I_{2}}(g_k)|\choose i-1}$$

$$=\beta_{i,j}(I_{2})+\beta_{i-1,j-1}(I_{2}).$$
Also $$\sum_{\deg(f_t)=j-i-1}{|\set_{I_{1}}(f_t)|\choose i}=\beta_{i,j-1}(I_{1}).$$
Therefore $$\beta_{i,j}(I)=\beta_{i,j-1}(I_{1})+\beta_{i,j}(I_{2})+\beta_{i-1,j-1}(I_{2}).$$
Moreover $I_2\subseteq I_1$ implies that $xI_{1}\cap I_{2}=xI_{2}$.
Using this equality and the equalities $\beta_{i,j-1}(I_{1})=\beta_{i,j}(xI_{1})$ and
$\beta_{i-1,j-1}(I_{2})=\beta_{i-1,j}(xI_{2})$, yields to
$$\beta_{i,j}(I)=\beta_{i,j}(xI_{1})+\beta_{i,j}(I_{2})+\beta_{i-1,j}(xI_{1}\cap I_{2}),$$
which completes the proof.
\end{proof}

\begin{rem}\label{fhvtuyl}
{\rm
Francisco, Ha and Van Tuyl in \cite{FHV} defined the concept of $x$-partition for a monomial ideal, when $x\in X$. For a monomial ideal $I$, if $J$ is the ideal generated by all elements of $\mathcal{G}(I)$ divisible by $x$, and $K$ is the ideal generated by all other elements of $\mathcal{G}(I)$, then $I=J+K$ is called an \textbf{$x$-partition}. If $I=J+K$ is also a Betti splitting, they call $I=J+K$ an \textbf{$x$-splitting}. In this regard, every vertex splitting is an $x$-splitting for some variable $x\in X$, considering Theorem \ref{ss}.
}
\end{rem}

\begin{rem}\label{bettis}
{\rm From the proof of Theorem \ref{ss}, one can see that for a vertex  splittable ideal $I$ with vertex splitting $I=xI_1+I_2$, the graded Betti numbers of $I$
can be computed by the following recursive formula
$$\beta_{i,j}(I)=\beta_{i,j-1}(I_{1})+\beta_{i,j}(I_{2})+\beta_{i-1,j-1}(I_{2}).$$}
\end{rem}

In the following corollary, the recursive formula is written for the graded Betti numbers of $I_{\Delta^{\vee}}$, when $\Delta$ is a vertex decomposable simplicial complex and consequently some inductive formulas for the regularity and projective dimension of the ring $R/I_{\Delta}$ are presented. The inductive formula given below for $\T{reg}(R/I_{\Delta})$ was also proved in \cite{HTW} by a different approach.

\begin{cor}\label{corvd}
Let $\Delta$ be a vertex decomposable simplicial complex, $x$ a shedding vertex of $\Delta$, $\Delta_1=\T{del}_{\Delta}(x)$
and $\Delta_2=\T{lk}_{\Delta}(x)$. Then
\begin{itemize}
  \item [(i)] $\beta_{i,j}(I_{\Delta^{\vee}})=\beta_{i,j-1}(I_{\Delta_1^{\vee}})+\beta_{i,j}(I_{\Delta_2^{\vee}})+\beta_{i-1,j-1}(I_{\Delta_2^{\vee}})$,
  \item [(ii)] $\T{pd}(R/I_{\Delta})=\max\{\T{pd}(R/I_{\Delta_1})+1,\T{pd}(R/I_{\Delta_2})\}$,
  \item [(iii)](Compare \cite[Theorem 4.2]{HTW}.) $\T{reg}(R/I_{\Delta})=\max\{\T{reg}(R/I_{\Delta_1}),\T{reg}(R/I_{\Delta_2})+1\}$.
\end{itemize}
\end{cor}
\begin{proof}
(i)  follows from Theorems \ref{vI} and \ref{ss}. (ii) and (iii) follow from (i), the equalities $\T{pd}(I_{\Delta^{\vee}})=\T{reg}(R/I_{\Delta})$ and $\T{reg}(I_{\Delta^{\vee}})=\T{pd}(R/I_{\Delta})$ in conjunction with
Theorem \ref{s}.
\end{proof}

\section{Applications to vertex cover ideal of a vertex decomposable graph}

This section is devoted to some applications of the recursive formulas presented in previous section to some special classes of graphs.
For a simple graph $G$ by $V(G)$ and $E(G)$ we mean the vertex set and the edge set of $G$, respectively. For a vertex $v\in V(G)$, set $N_G(v)=\{u\in V(G):\ \{u,v\}\in E(G)\}$ and $N_G[v]=N_G(v)\cup \{v\}$. Moreover, the cardinality of $N_G(v)$ is called the degree of $v$ in $G$ and is denoted by $\deg_G(v)$.

The following is one of our main results which is a consequence of Corollary \ref{corvd}.

\begin{thm}\label{graphvd}
Let $G$ be a vertex decomposable graph, $v\in V(G)$ be a shedding vertex of $G$, $G'=G\setminus \{v\}$,
$G''=G\setminus N_G[v]$ and $\deg_G(v)=t$. Then
$$\beta_{i,j}(I(G)^{\vee})=\beta_{i,j-1}(I(G')^{\vee})+\beta_{i,j-t}(I(G'')^{\vee})+\beta_{i-1,j-t-1}(I(G'')^{\vee}).$$
\end{thm}

\begin{proof}
Let $\Delta=\Delta_G$, $\Delta_1=\T{del}_{\Delta}(v)$, $\Delta_2=\T{lk}_{\Delta}(v)$ and $N_G(v)=\{x_1,\ldots,x_t\}$. Then $I_{\Delta_G^{\vee}}=(I_{\Delta_G})^{\vee}=I(G)^{\vee}$. Moreover, $\Delta_{G'}=\{F\in \Delta:\ F\subseteq V(G')\}=\{F\in \Delta:\ v\notin F\}=\Delta_1$. Thus $I_{\Delta_1^{\vee}}=I_{\Delta_{G'}^{\vee}}=(I_{\Delta_{G'}})^{\vee}=I(G')^{\vee}$.

Also $I_{\Delta_2^{\vee}}=(x^{(V(G)\setminus \{v\})\setminus F}:\ F \in \mathcal{F}(\Delta_2))$, since $\Delta_2$ is a simplicial complex on the vertex set $V(G)\setminus \{v\}$. Moreover, $F\in \mathcal{F}(\Delta_2)$ if and only if $v,x_1,\ldots,x_t\notin F$ and $F\in \mathcal{F}(\Delta_{G''})$. This also implies that for any $F\in \mathcal{F}(\Delta_2)$, $F\subseteq V(G'')$. Thus $(V(G)\setminus \{v\})\setminus F=\{x_1,\ldots,x_t\}\cup (V(G'')\setminus F)$.
Therefore $$I_{\Delta_2^{\vee}}=x_1\cdots x_t\ (x^{V(G'')\setminus F}:\ F\in \mathcal{F}(\Delta_{G''})).$$
Moreover, $(x^{V(G'')\setminus F}:\ F\in \mathcal{F}(\Delta_{G''}))=(I_{\Delta_{G''}})^{\vee}=I(G'')^{\vee}$.
Thus $I_{\Delta_2^{\vee}}=x_1\cdots x_t I(G'')^{\vee}$.
Using the equality $\beta_{i,j}(x_1\cdots x_tI(G'')^{\vee})=\beta_{i,j-t}(I(G'')^{\vee})$,
the result is now clear from Corollary \ref{corvd}.
\end{proof}

By exploiting the following lemma, we state the recursive formula for the graded Betti numbers of the vertex cover ideal
of a sequentially Cohen-Macaulay bipartite graph, which generalizes \cite[Theorem 3.8]{FHV} (since any Cohen-Macaulay bipartite graph is sequentially Cohen-Macaulay), and vertex cover ideal of
chordal graphs in Corollaries \ref{seq} and \ref{chor}.

\begin{lem}\cite[Lemma 6]{W}\label{lshedding}
Let $G$ be a graph and $x,y\in V(G)$. If $N_G[x]\subseteq N_G[y]$, then $y$ is a shedding vertex for $G$.
\end{lem}

A graph $G$ with vertex set $V(G)$ is called \textbf{bipartite} if $V(G)$ can be partitioned into two sets $X$ and $Y$ such that any edge of $G$ is of the form $\{x,y\}$ for some $x\in X$ and $y\in Y$.

Van Tuyl and Villarreal in \cite{VVi} gave a recursive characterization for a sequentially Cohen-Macaulay bipartite graph as follows.

\begin{thm}\cite[Corollary 3.11]{VVi}
Let $G$ be a bipartite graph. Then $G$ is sequentially Cohen-Macaulay if and only if there are
adjacent vertices $x$ and $y$ with $\deg_G(x)=1$ such that the
bipartite graphs $G'=G\setminus N_G[x]$ and $G''=G\setminus N_G[y]$ are sequentially Cohen-Macaulay.
\end{thm}
Using this fact we can explain the formula for Betti numbers as follows.
\begin{cor}\label{seq}
Let $G$ be a sequentially Cohen-Macaulay bipartite graph, $x,y\in V(G)$ be adjacent vertices with $\deg_G(x)=1$
such that $G'=G\setminus N_G[x]$ and $G''=G\setminus N_G[y]$ are sequentially Cohen-Macaulay and $\deg_G(y)=t$. Then
$$\beta_{i,j}(I(G)^{\vee})=\beta_{i,j-1}(I(G')^{\vee})+\beta_{i,j-t}(I(G'')^{\vee})+\beta_{i-1,j-t-1}(I(G'')^{\vee}).$$
\end{cor}

\begin{proof}
From \cite[Theorem 2.10]{VT}, we know that any sequentially Cohen-Macaulay bipartite graph $G$ is vertex decomposable.
Since $N_G[x]\subseteq N_G[y]$,
by Lemma \ref{lshedding}, $y$ is a shedding vertex for $G$. Now the result is clear by Theorem \ref{graphvd} and the fact that $I(G')^{\vee}=I(G\setminus \{y\})^{\vee}$.
\end{proof}

In a graph $G$, a vertex $x$ is called a \textbf{simplicial vertex} if the induced subgraph on the set $N_G[x]$ is a complete graph.
A graph $G$ is called \textbf{chordal}, if it contains no induced cycle of length $4$ or greater.

Dirac in \cite{Dirac} proved that any chordal graph has a simplicial vertex.  We use this fact in the following corollary.

\begin{cor}\label{chor}
Let $G$ be a chordal graph with simplicial vertex $x$ and $y\in N_G(x)$ with $\deg_G(y)=t$. Let $G'=G\setminus \{y\}$ and $G''=G\setminus N_G[y]$. Then
$$\beta_{i,j}(I(G)^{\vee})=\beta_{i,j-1}(I(G')^{\vee})+\beta_{i,j-t}(I(G'')^{\vee})+\beta_{i-1,j-t-1}(I(G'')^{\vee}).$$
\end{cor}

\begin{proof}
By \cite[Corollary 7]{W} any chordal graph is vertex decomposable. Since $x$ is a simplicial vertex, for any $y\in N_G(x)$, we have $N_G[x]\subseteq N_G[y]$. Thus by Lemma \ref{lshedding},
 $y$ is a shedding vertex for $G$. Now apply Theorem \ref{graphvd}.
\end{proof}

The following theorem investigates another property of chordal graphs.
\begin{thm}\label{propchor}
Let $G$ be a chordal graph. Then $I(G^c)$ is a vertex splittable ideal.
\end{thm}
\begin{proof}
We prove the result by the induction on $|V(G)|$.
Let $x\in V(G)$ be a simplicial vertex of $G$ and $V(G)=\{x,x_1,\ldots,x_n\}$. If $G$ is a complete graph, then the result is clear. Assume that $G$
  is not a complete graph and without loss of generality let $N_{G}(x)=\{x_1,\ldots,x_k\}$ for some $1\leq k<n$ and $G_0=G\setminus \{x\}$.
Then $I(G^c)=x(x_{k+1},\ldots,x_n)+I(G^c_0)$. Moreover, for any distinct integers $i$ and $j$ with $1\leq i,j\leq k$, since $\{x_i,x_j\}\in E(G)$, then $x_ix_j\notin I(G^c_0)$. So $I(G^c_0)\subseteq (x_{k+1},\ldots,x_n)$. Since $G_0$ is chordal, by induction hypothesis
$I(G^c_0)$ is vertex splittable. Also it is easy to see that any ideal which is generated
 by variables is a vertex splittable ideal. Thus $(x_{k+1},\ldots,x_n)$ is vertex splittable. So $I(G^c)$ is a vertex splittable ideal as desired.
\end{proof}

Edge ideals with linear resolution were characterized in \cite{Fro} as follows.

\begin{thm} \cite[Theorem 1]{Fro}\label{Frochor}
For a graph $G$, the edge ideal $I(G)$ has linear resolution if and only if $G^c$ is a chordal graph.
\end{thm}

Mohammadi in \cite{Mohammadi}  proved that for a chordal graph $G$, $\Delta(G)^{\vee}$ is vertex decomposable, where $\Delta(G)$ is the
clique complex of $G$. By means of Theorem \ref{propchor}, we are able to give another proof of this result.

\begin{cor}\label{corchor1}
For a graph $G$ the followings are equivalent.
\begin{itemize}
\item[(i)] $G^c$ is chordal,
\item[(ii)] $I(G)$ is vertex splittable,
\item[(iii)] $\Delta_G^{\vee}$ is vertex decomposable,
\item[(iv)] $\Delta_G^{\vee}$ is Cohen-Macaulay.
\end{itemize}
\end{cor}
\begin{proof}
(i)$\Rightarrow$ (ii)  is the statement of Theorem \ref{propchor}.

 (ii)$\Rightarrow$ (iii) follows from Theorem \ref{vI}, noting the fact that $I_{\Delta_G}=I(G)$ and $(\Delta_G^{\vee})^{\vee}=\Delta_G$.

(iii)$\Rightarrow$ (iv) Since $\Delta_G^{\vee}=\langle V(G)\setminus \{x,y\}:\ \{x,y\}\in E(G)\rangle$, it is a pure vertex decomposable simplicial complex and so it is Cohen-Macaulay.

(iv)$\Rightarrow$ (i) By Theorem \ref{EG}, $I_{\Delta_G}=I(G)$ has a linear resolution. So by Theorem \ref{Frochor}, $G^c$ is a chordal graph.
\end{proof}

\providecommand{\bysame}{\leavevmode\hbox
to3em{\hrulefill}\thinspace}


\begin{thebibliography}{10}

\bibitem{BFHV} {\sc J. Biermann; C. A. Francisco; H. T. H\`{a}; A. Van Tuyl ,} {\it Colorings of simplicial complexes and vertex decomposability.} Preprint, arXiv:math.AC/1209.3008v1.

\bibitem{BV} {\sc J. Biermann; A. Van Tuyl ,} {\it Balanced vertex decomposable simplicial complexes and their h-vectors.} Electron. J. Combin. {\bf 20} (2013), no. 3, 15 pp.

\bibitem{BW} {\sc A. Bj\"{o}rner; M. L. Wachs,} {\it Shellable nonpure complexes and
posets I.} Trans. Amer. Math. Soc. {\bf 348} (1996) no. 4,
1299–-1327.

\bibitem{BW2} {\sc A. Bj\"{o}rner; M. L. Wachs,} {\it Shellable nonpure complexes and posets. II.} Trans. Amer. Math. Soc. 349 (1997), no. 10, 3945--3975.

\bibitem{CN} {\sc D. Cook II; U. Nagel,}  {\it Cohen-Macaulay graphs and face vectors of flag complexes.} SIAM J. Discrete
Math. 26 (2012), no. 1, 89–-101.

\bibitem{Dirac} {\sc G.A. Dirac,} {\it On rigid circuit graphs.}  Abh. Math. Sem. Univ. Hamburg 24 (1961), 71-76.

\bibitem{DE} {\sc A. Dochtermann; A. Engstr\"{o}m,} {\it Algebraic properties of edge ideals via combinatorial topology.} Electron. J. Combin. {\bf 16} (2009), no. 2, Special volume in honor of Anders Bjorner, Research Paper 2, 24 pp.

\bibitem{Eagon} {\sc J. A. Eagon; V. Reiner,} {\it Resolutions of Stanley-Reisner rings and Alexander duality.} J. Pure Appl. Algebra 130 (1998), no. 3, 265--275.

\bibitem{F} {\sc S. Faridi,} {\it Simplicial trees are sequentially Cohen-Macaulay.} J. Pure Appl. Algebra 190 (2004), no. 1-3, 121--136.

\bibitem{FHV} {\sc C. A. Francisco; H. T. H\`{a}; A. Van Tuyl,} {\it Splittings of monomial ideals.} Proc. Amer. Math. Soc. 137 (2009), no. 10, 3271--3282.

\bibitem{Fro} {\sc R. Fr\"{o}berg,} {\it On Stanley-Reisner rings.} Topics in Algebra, Banach Center Publications, 26 (1990), 57--70.


\bibitem{HTW} {\sc H. T. H\`{a}; R. Woodroofe,} {\it Results on the regularity of square-free monomial ideals.} Preprint, arXiv:math.AC/1301.6779v1.

\bibitem{HD} {\sc J. Herzog; T. Hibi; X. Zheng,} {\it Dirac's theorem on chordal graphs and Alexander duality.} European J. Combin. 25 (2004), no. 7, 949--960.

\bibitem{HH} {\sc J. Herzog; T. Hibi,} {\it  Componentwise linear ideals.} Nagoya Math. J. 153 (1999), 141--153.

\bibitem{KM} {\sc F. Khosh-Ahang; S. Moradi,}  {\it Regularity and projective dimension of edge ideal of $C_5$-free vertex decomposable graphs.} To appear in Proc. Amer. Math. Soc.

\bibitem{Mohammadi} {\sc F. Mohammadi,}   {\it Powers of the vertex cover ideal of a chordal graph.} Comm. Algebra 39 (2011), no. 10, 3753--3764.

\bibitem{Moradi} {\sc S. Moradi; D. Kiani,}  {\it Bounds for the regularity of edge ideals of vertex decomposable
 and shellable graphs.} Bull. Iranian Math. Soc. 36 (2010), no. 2, 267-277.


\bibitem{Provan+Billera} {\sc J. S. Provan; L. J. Billera,} {\it Decompositions of simplicial complexes related to diameters of convex polyhedra.} Math. Oper. Res. 5 (1980), no. 4, 576--594.

\bibitem{RY}{\sc R. Rahmati-Asghar; S. Yassemi,}  {\it k-decomposable monomial ideals.} Preprint.


\bibitem{Leila} {\sc L. Sharifan; M. Varbaro,} {\it Graded Betti numbers of ideals with linear quotients.} Le Matematiche (Catania) 63 (2008), no. 2, 257--265.

\bibitem{T} {\sc N. Terai,} {\it Alexander duality theorem and Stanley-Reisner rings. Free resolutions of coordinate rings of projective varieties and related topics} (Japanese) (Kyoto, 1998). S\"{u}rikaisekikenky\"{u}sho K\"{o}ky\"{u}roku no. 1078 (1999), 174--184.

\bibitem{VT} {\sc A. Van Tuyl,} {\it Sequentially Cohen-Macaulay bipartite graphs: vertex decomposability and regularity.} Arch. Math. (Basel) 93 (2009), no. 5, 451--459.

\bibitem{VVi} {\sc A. Van Tuyl; R. Villarreal,} {\it Shellable graphs and sequentially Cohen-Macaulay bipartite graphs.} J.
Combin. Theory Ser. A 115 (2008), no. 5, 799â€“-814.

\bibitem{W} {\sc R. Woodroofe,} {\it Vertex decomposable graphs and obstructions to shellability.} Proc. Amer. Math. Soc. 137 (2009), no. 10, 3235--3246.

\bibitem{W1} {\sc R. Woodroofe,} {\it Chordal and sequentially Cohen-Macaulay clutters.} Electron. J. Combin. 18 (2011),
no. 1, Paper 208, 20 pp.

\end{thebibliography}
\end{document}